\documentclass[11pt, reqno, twoside]{amsart}

\usepackage[top=1in,bottom=1in,left=1.2in,right=1.2in]{geometry}

\usepackage{amsmath,amssymb,amsthm,amsfonts,enumerate, mathrsfs, mathtools, bm}
\usepackage{appendix}
\usepackage{relsize}
\usepackage{stmaryrd}
\usepackage{esint}
\usepackage{xcolor}
\usepackage{graphicx}
\usepackage{tikz}
\usetikzlibrary{patterns.meta}
\usepackage[numbers]{natbib}

\usepackage[
colorlinks=true,
linkcolor=blue,
filecolor=blue,
urlcolor=red,
citecolor=magenta]{hyperref}
\usepackage[nameinlink]{cleveref}

\newcounter{counterConstant}

\def\cF{{\mathcal F}}

\def\mN{{\mathbb N}}

\def\mR{{\mathbb R}}
\def\mS{{\mathbb S}}

\def\mZ{{\mathbb Z}}

\def\sS{{\mathscr S}}

\def\geq{\geqslant}
\def\leq{\leqslant}

\def\1{{\mathbf{1}}}
\def\p{\partial}
\def\d{\text{\rm{d}}}
\def\e{\mathrm{e}}
\def\eps{\varepsilon}

\def\esssup{\mathop{\mathrm{ess\,sup}}}

\numberwithin{equation}{section}

\newtheorem{theorem}{Theorem}[section]
\newtheorem{lemma}[theorem]{Lemma}
\newtheorem{remark}[theorem]{Remark}

\usepackage{fancyhdr}
{}

\allowdisplaybreaks

\title{On the Schauder Estimates for Non-local Equations with Drift: The Supercritical Case}
\author{Yanfang Li and Guohuan Zhao}
\date{}

\address{College of Mathematics and Physics, Beijing University of Chemical Technology, Beijing,100029, China}
\email{liyanfang@buct.edu.cn}

\address{ State Key Laboratory of Mathematical Sciences, Academy of Mathematics and Systems Science, CAS, Beijing, 100190, China}
\email{gzhao@amss.ac.cn}

\thanks{The research of Yanfang Li is supported by the Fundamental Research Funds for the Central Universities (ZY2628). The research of Guohuan Zhao is supported by the National Key Research and Development Program of China (No. 2024YFA1013503) and the National Natural Science Foundation of China grants (No. 12271352).}

\begin{document}

\begin{abstract}
We establish a Schauder estimate for a nonlocal Cauchy problem with drift. The leading operator is the generator of a non-degenerate $\alpha$-stable process with $\alpha\in(0,1)$, and the drift is $\beta$-Hölder continuous with $\beta \in (1-\alpha,1)$. The proof relies on a refined conic Littlewood--Paley decomposition and a one-dimensional one-sided dyadic maximum principle.
\end{abstract}

\maketitle



\section{Introduction and main estimate}

The main purpose of this paper is to establish a Schauder estimate for the following non-local parabolic equation
\begin{equation}\label{eq:PDE}
    \p_t u+\lambda u-Lu-b\cdot\nabla u=h,
    \quad u(0)=f.
\end{equation}
Here  $b: [0,T] \times \mR^d \to \mR^d$, and $L$ is the infinitesimal generator of an $\alpha$-stable process $Z=(Z_t)_{t \geq 0}$ with $\alpha \in (0,1)$. More precisely,  \begin{equation*}
    L\varphi(x)
    :=
    \int_{\mR^d}
    \bigl(\varphi(x+z)-\varphi(x)\bigr)\nu(\d z), \quad \varphi\in C_b^1(\mR^d), 
\end{equation*}
where $\nu$ is the L\'evy measure of $Z$. Equation \eqref{eq:PDE} is the backward Kolmogorov equation associated with the SDE 
\begin{equation}\label{eq:SDE}
    \d X_t=b(t,X_t)\,\d t+\d Z_t,
    \qquad X_0=x.
\end{equation}

The $\alpha$-stable Lévy measure $\nu$ admits the polar decomposition
\[
    \nu(A)
    =
    \int_{\mS^{d-1}}\int_0^\infty
    \1_A(r\theta)\,\frac{\d r}{r^{1+\alpha}}\,\mu(\d\theta),
    \quad A\subset \mR^d\setminus\{0\}, 
\]
where $\mu$ is a finite non-negative Borel measure on $\mS^{d-1}$, called
the spectral measure of $Z$. 

\medskip

Thoughout this paper, we impose the non-degeneracy condition
\begin{equation}\label{eq:mu}
    \int_{\mS^{d-1}} |e\cdot\theta|^\alpha\,\mu(\d\theta)
    \geq \kappa>0,
    \qquad e\in\mS^{d-1}. 
    \tag{{\bf A}}
\end{equation}

Our main result is the following.

\begin{theorem}\label{thm:main}
Let $\alpha\in (0,1)$, $\beta\in(1-\alpha,1)$ and $s:=\alpha+\beta\in(1,2)$. Assume \eqref{eq:mu} holds. Suppose that 
\[
    \| b \|_{L^\infty_T C^\beta_x}, ~ \| h \|_{L^\infty_T C^\beta_x}, ~ \|f\|_{C^s}<\infty, 
\]
and $u$ is a smooth bounded solution of \eqref{eq:PDE}, then
\begin{equation}\label{eq:main-pre}
    [u]_{L_T^\infty \dot{C}_x^s}
    \leq
    C\left(
        [f]_{\dot{C}^s}
        +[h]_{L_T^\infty \dot{C}_x^\beta}
        +[b]_{L_T^\infty \dot{C}_x^\beta}
        \|\nabla u\|_{L^\infty_T L^\infty_x}
    \right), 
\end{equation}
where $C$ depends only on $d,\alpha,\beta,\kappa,\mu(\mS^{d-1})$. Consequently, for every
$\lambda>0$,
\begin{equation}\label{eq:main-final}
    \|u\|_{L_T^\infty C_x^s}
    \leq
    C\left(
        1+[b]_{L_T^\infty \dot{C}_x^\beta}^{s/(s-1)}
    \right)\|f\|_{C^s}
    +C\left(
        1+\lambda^{-1}
        +\lambda^{-1}[b]_{L_T^\infty \dot{C}_x^\beta}^{s/(s-1)}
    \right)\|h\|_{L_T^\infty C_x^\beta}.
\end{equation}
\end{theorem}

\begin{remark}
\begin{enumerate}
    \item The above theorem is an a priori estimate. It is stated for smooth
    bounded solutions in order to focus on the proof of the estimate. More
    general versions can be obtained by mollifying the data and the solution,
    proving the estimate uniformly, and then passing to the limit.
    \item The assumption that $L$ is translation-invariant and exactly
    stable is made only to streamline the presentation. The argument can
    also be extended to time- and space-dependent L\'evy-type operators
    of the form
    \[
        L_t\varphi(x)
        :=
        \int_{\mR^d}
        \bigl(\varphi(x+z)-\varphi(x)\bigr)\nu_{t,x}(\d z),
    \]
    where the jump measure $\nu_{t,x}$ is bounded between the L\'evy measures of two non-degenerate $\alpha$-stable
    processes and has sufficient regularity assumptions in $x$. 
\end{enumerate}
\end{remark}

Regularity estimates for non-local operators with irregular drift play an
important role in the study of nonlinear integro-differential equations and stochastic differential equations driven by L\'evy processes. In the model case of the fractional Laplacian, the interplay between the regularizing effect of non-local diffusion and the drift has been studied through Schauder, H\"older, and related regularity estimates for integro-differential equations; see for instance, 
\cite{silvestre2006holder,caffarelli2010drift,silvestre2012differentiability,chen2014holder}. Regularity theories have also been developed for non-local equations with non-translation-invariant,  anisotropic, or singular jump kernels; see, among others, \cite{kassmann2009priori,dong2018dini,dyda2020regularity,DRSV22nonsymmetric,FR24schauder}.

For equation \eqref{eq:PDE}, the regime $\alpha\in(0,1)$ is supercritical because the first-order drift $b\cdot\nabla$ has higher differential order than the non-local diffusion $L$. For the fractional Laplacian
$L=\Delta^{\alpha/2}$, Silvestre's work
\cite{silvestre2012differentiability} identified the condition
$\alpha+\beta>1$ as the natural threshold for obtaining differentiability of the solution. This threshold also appears in the study of  SDE \eqref{eq:SDE} with irregular drift $b$. In particular, Priola proved pathwise
uniqueness for SDEs of the form \eqref{eq:SDE} with H\"older drifts
\cite{priola2012pathwise,priola2015stochastic}, using Schauder estimates for
the associated Kolmogorov equation \eqref{eq:PDE} together with Zvonkin's transformation.
Such estimates are also useful in the study of stochastic flows property, path-by-path uniqueness for random ODEs, and convergence rates of the Euler--Maruyama scheme for \eqref{eq:SDE}, where one needs sufficiently regular solutions of the Kolmogorov equation to control the
transformed dynamics (see for instance \cite{priola2018davie} and \cite{LZ24euler}). 

However, available H\"older-space Schauder estimates for \eqref{eq:PDE} in supercritical case are more restricted beyond the isotropic fractional Laplacian. de Raynal, Menozzi and Priola 
\cite{de2020schauder} treated general stable-type generators
satisfying suitable heat-kernel smoothing assumptions. For general
non-degenerate symmetric stable operators, including the cylindrical stable
operator, the verification in
\cite{de2020schauder} requires $\beta<\alpha$; together with
$\alpha+\beta>1$, this imposes the restriction $\alpha>1/2$. Other partial
results were obtained, for instance, in
\cite{chen2018stochastic,HW24schauder,LZ22nonlocal}, under additional
restrictions on the range of $\alpha$ or on the structure of the operator. Related estimates in Besov or fractional Sobolev spaces were obtained in
\cite{chen2021supercritical,zhao2021regularity}. We also mention that a closely related estimate was stated in \cite[Theorem 4.8]{SX23weak}. Its proof uses the frequency-localized maximum principle of \cite{WZ11frequency}, but whose direct
form does not apply to the general anisotropic stable measures considered
there and here. The angular maximum principle developed below will is designed to address this point.

\medskip

In conclusion, these restrictions leave a gap in the
H\"older-space Schauder theory for general
stable operators in the whole supercritical range $\alpha\in(0,1)$. The purpose of this paper is to prove such an estimate for the parabolic equation \eqref{eq:PDE}, where
the leading operator is generated by an arbitrary non-degenerate $\alpha$-stable L\'evy measure, and $b\in C^\beta$ with $\beta>1-\alpha$. The proof is based on a fine conic Littlewood--Paley decomposition. The main observation is that the stable
L\'evy measure yields an angular dyadic maximum principle, which plays
the role of the usual dyadic maximum principle for the isotropic fractional
Laplacian. This point is especially useful for anisotropic or singular jump
measures, where the standard isotropic maximum-principle argument is no longer
directly available, and is the reason that the same method can be extended to
more general L\'evy operators as indicated in the remark above.

\medskip

We close this section by mentioning some notations and recalling very basic facts from standard Littlewood-Paley theory: Let $\sS(\mR^d)$ be the Schwartz space of all rapidly decreasing functions, and $\sS'(\mR^d)$ the dual space of $\sS(\mR^d)$ 
called Schwartz generalized function (or tempered distribution) space. For any $f\in \sS(\mR^d)$ define its Fourier transform by
\[
\cF(f)(\xi)=\widehat{f}(\xi)=\int_{\mR^d} \e^{-i x \cdot \xi} f(x)\, \d x. 
\]

Let $\chi:\mR^d\to[0,1]$ be a smooth radial function so that $\chi|_{B_{3/4}}=1$ and $\chi|_{B_1^c}=0$. Define
$$
\varphi(\xi):=\chi(\xi)-\chi(2\xi).
$$
For $j\in\mZ$, define the homogeneous dyadic block by
\[
\dot{\Delta}_j f
:=
\cF^{-1}\left(
    \varphi(2^{-j}\,\cdot)\widehat f
\right).
\]

Let $\mathcal P$ denote the space of polynomials on $\mR^d$.  Homogeneous Besov spaces are naturally defined on the quotient
$\sS'(\mR^d)/\mathcal P$. For $s\in\mR$, set
\[
\dot B^s_{\infty,\infty}(\mR^d)
:=
\Big\{
f\in \sS'(\mR^d)/\mathcal P:
\|f\|_{\dot B^s_{\infty,\infty}}:=
\sup_{j\in\mZ}
2^{js}\|\dot{\Delta}_j f\|_{L^\infty}<\infty\Big\}. 
\]

We next recall the homogeneous H\"older spaces. Let $0<s\notin \mN$ and write $s=k+\gamma$ with $k \in \mN$ and $\gamma\in(0,1)$. Denote by $\mathcal P_k$ the space of polynomials on $\mR^d$ of degree at
most $k$. Define 
\[
\dot C^s(\mR^d)
:=
\Big\{
f\in C^k_{\mathrm{loc}}(\mR^d):
[f]_{\dot C^s}
:= 
\sup_{x\neq y}
\frac{|\nabla^k f(x)-\nabla^k f(y)|}
{|x-y|^\gamma}
<\infty
\Big\}\big/\mathcal P_k.
\]
It is well-known that for every
positive non-integer $s$, 
\[
    \dot B^s_{\infty,\infty}(\mR^d)=\dot C^s(\mR^d),
    \qquad
    \|f\|_{\dot B^s_{\infty,\infty}}
    \asymp [f]_{\dot C^s}.
\]
See, for instance, \cite{bahouri2011fourier}.

The inhomogeneous H\"older space is 
\[
    C^s(\mR^d)
    :=
    \Big\{
    f\in C^k(\mR^d):
    \|f\|_{C^s}
    :=
    \sum_{0 \leq i \leq k}\|\nabla^i f\|_{L^\infty}
    +[f]_{\dot C^s}
    <\infty
    \Big\}.
\]
For a function $f=f(t,x)$ on $[0,T]\times\mR^d$, we set
\[
    \|f\|_{L_T^\infty C_x^s}
    :=
    \esssup_{t\in[0,T]}\|f(t,\cdot)\|_{C^s},
    \qquad
    [f]_{L_T^\infty\dot C_x^s}
    :=
    \esssup_{t\in[0,T]}[f(t,\cdot)]_{\dot C^s}.
\]

\section{Maximum principles}

\subsection{One-dimensional dyadic maximum principle}

We begin with a one-dimensional one-sided dyadic maximum principle.

\begin{lemma}\label{lem:1d-dmp}
Let $0<\alpha<1$ and $0<a<A<\infty$. There exists
$c=c(\alpha,a,A)>0$ such that the following holds. Let $g:\mR\to\mR$ be
bounded and real-valued, and suppose
\[
    \operatorname{supp} \widehat g
    \subset
    \{\eta\in\mR: a2^j\leq |\eta|\leq A2^j\}
\]
for any $j\in \mZ$. If
\[
    g(0)=\|g\|_{L^\infty}>0,
\]
then
\begin{equation*}
    \int_0^\infty \frac{g(0)-g(r)}{r^{1+\alpha}}\,\d r
    \geq c2^{j\alpha} g(0).
\end{equation*}
\end{lemma}

\begin{proof}
By scaling and normalization it suffices to prove the claim for $j=0$ and
$g(0)=\|g\|_{L^\infty}=1$. Suppose the claim is false. Then there exists a
sequence $g_n$ such that
\[
    \operatorname{supp}\widehat g_n\subset\{a\leq |\eta|\leq A\},
    \quad
    g_n(0)=\|g_n\|_{L^\infty}=1,
\]
but
\[
    \int_0^\infty \frac{1-g_n(r)}{r^{1+\alpha}}\,\d r\to0.
\]
Since the Fourier supports are contained in a fixed compact set and
$\|g_n\|_{L^\infty}\leq1$, Bernstein's inequalities imply uniform bounds for
all derivatives on compact sets. Passing to a subsequence, $g_n\to g$ locally
uniformly, with
\[
    g(0)=\|g\|_{L^\infty}=1,
    \quad
    \operatorname{supp}\widehat g\subset\{a\leq |\eta|\leq A\}.
\]
Since $1-g_n(r)\geq0$, Fatou's lemma gives, for every $R>0$,
\[
    \int_0^R \frac{1-g(r)}{r^{1+\alpha}}\,\d r=0.
\]
Hence $g(r)=1$ for every $r>0$. The function $g$ is band-limited, hence real
analytic; therefore $g\equiv1$ on $\mR$. This contradicts the fact that the
Fourier support of $g$ is disjoint from the origin. The lemma follows.
\end{proof}

Next we recall a standard support property for restrictions to lines.

\begin{lemma}\label{lem:support}
Let $v$ be a smooth tempered function on $\mR^d$ such that
$\operatorname{supp}\widehat v\subset K$, where $K\subset\mR^d$ is compact. Fix
$x_0\in\mR^d$ and $\theta\in\mS^{d-1}$, and define
\[
    g_\theta(r):=v(x_0+r\theta),
    \quad r\in\mR.
\]
Then, as a distribution on $\mR$,
\[
    \operatorname{supp} \widehat{g_\theta}
    \subset
    \{\xi\cdot\theta:\xi\in K\}.
\]
\end{lemma}

\begin{proof}
Since $\widehat v$ is supported in the compact set $K$, the inverse Fourier formula gives
\[
\begin{aligned}
    g_\theta(r)=(2\pi)^{-d}
    \left\langle
        \widehat v(\xi),
        e^{ix_0\cdot\xi}e^{ir\xi\cdot\theta}
    \right\rangle.
\end{aligned}
\]
Let $\varphi\in C_c^\infty(\mR)$ and define its one-dimensional Fourier
transform by $\widehat\varphi(r):=\int_{\mR}e^{-ir\eta}\varphi(\eta)\,\d\eta$. By the definition of the Fourier transform of a distribution, we have
\begin{equation}\label{eq:int_g_phi}
\begin{aligned}
    \langle \widehat{g_\theta},\varphi\rangle
    &=\langle g_\theta,\widehat\varphi\rangle=(2\pi)^{-d}
    \left\langle
        \widehat v(\xi),
        e^{ix_0\cdot\xi}
        \int_{\mR}e^{ir\xi\cdot\theta}
        \widehat\varphi(r)\,\d r
    \right\rangle \\
    &=(2\pi)^{1-d}
    \left\langle
        \widehat v(\xi),
        e^{ix_0\cdot\xi}\varphi(\xi\cdot\theta)
    \right\rangle.
\end{aligned}
\end{equation}
In the last equality we used the one-dimensional inversion identity
\[
    \int_{\mR}e^{irs}\widehat\varphi(r)\,\d r
    =2\pi\varphi(s).
\]

Now suppose that
\[
    \operatorname{supp}\varphi\cap \{\xi \cdot \theta: \xi\in K \}=\varnothing.
\]
Since $\operatorname{supp}\widehat v\subset K$, and $\xi\mapsto\varphi(\xi\cdot\theta)$ vanishes on a neighbourhood of
$K$, the last integration in \eqref{eq:int_g_phi} is zero. Hence, one sees that 
\[
    \operatorname{supp}\widehat{g_\theta}
    \subset \{\xi\cdot\theta:\xi\in K\}.
\]
\end{proof}

\subsection{Conic decomposition and Angular dyadic maximum principle}

Let $M_\mu:=\mu(\mS^{d-1})$. Choose $\delta\in(0,1)$ so small that
\begin{equation}\label{eq:delta}
    (2\delta)^\alpha M_\mu\leq \frac{\kappa}{2}.
\end{equation}
For $e\in\mS^{d-1}$, set
\[
    E(e):=\{\theta\in\mS^{d-1}: |e\cdot\theta|\geq 2\delta\}.
\]
By \eqref{eq:mu} and \eqref{eq:delta}, we have 
\[
\begin{aligned}
    \kappa \leq \int_{\mS^{d-1}} |e\cdot\theta|^\alpha\,\mu(\d\theta)  \leq (2\delta)^\alpha M_\mu + \mu(E(e))\leq \frac{\kappa}{2}+\mu(E(e)), 
\end{aligned}
\]
which yields that 
\begin{equation}\label{eq:E-lower}
    \mu(E(e))\geq \frac{\kappa}{2},
    \quad e\in\mS^{d-1}.
\end{equation}
Choose finitely many points $e_m\in\mS^{d-1}$, $1\leq m\leq N$, and smooth
real-valued even functions $q_m$, $1\leq m\leq N$, on $\mS^{d-1}$ such that
\[
    \bigcup_{m=1}^N E(e_m) = \mS^{d-1}; 
    \quad \sum_{m=1}^N q_m(\omega)=1,
    \quad \omega\in\mS^{d-1}; 
\]
and
\[
    \operatorname{supp} q_m
    \subset
    \{\omega\in\mS^{d-1}: \operatorname{dist}(\omega,\pm e_m)<\delta\}.
\]
Then, for every $\omega\in\operatorname{supp} q_m$ and $\theta\in E(e_m)$,
\begin{equation}\label{eq:cone-pro}
    |\omega\cdot\theta|\geq\delta.
\end{equation}

Recall that $(\dot{\Delta}_j)_{j\in\mZ}$ is the standard  Littlewood--Paley family with
\begin{equation}\label{eq:support}
    \operatorname{supp} \widehat{\dot{\Delta}_j h}
    \subset
    \{\xi\in\mR^d: c_0 2^j\leq |\xi|\leq C_0 2^j\},
    \quad j\in \mZ.
\end{equation}
Extend $q_m$ to $\mR^d\setminus\{0\}$ homogeneously of degree zero, and define the conic dyadic projectors
\begin{equation*}
    P_{j,m}h:=q_m(D)\dot{\Delta}_j h.
\end{equation*}
Since the $q_m$ are real and even, $P_{j,m}$ maps real-valued functions to
real-valued functions. Moreover,
\begin{equation}\label{eq:sum-Pjm}
    \dot{\Delta}_j h=\sum_{m=1}^N P_{j,m}h.
\end{equation}

Figure~\ref{fig:conic} illustrates the conic dyadic
decomposition in the 2D cylindrical case.
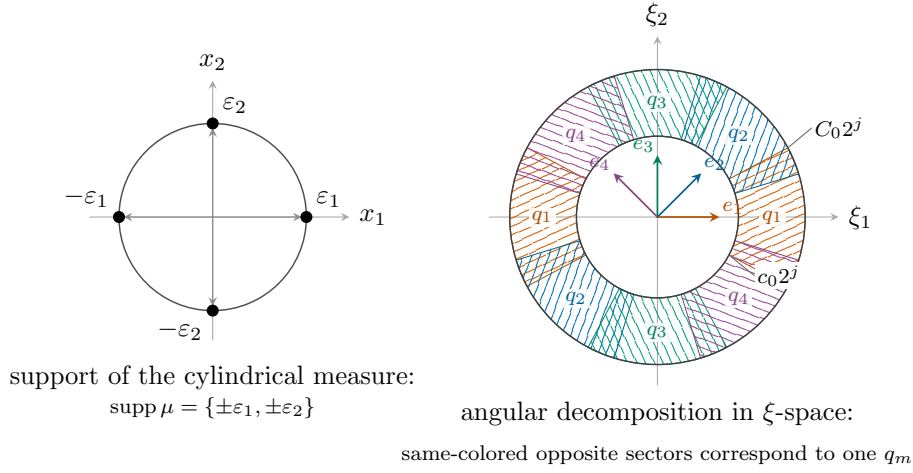
\begin{figure}[htbp]
    \centering
    \begin{tikzpicture}[
        scale=1.05,
        font=\small,
        line cap=round,
        line join=round,
        >=stealth
    ]
    \definecolor{angularorange}{HTML}{D55E00}
    \definecolor{angularblue}{HTML}{0072B2}
    \definecolor{angulargreen}{HTML}{009E73}
    \definecolor{angularpurple}{HTML}{9C4F96}
    \tikzset{
        qone/.style={
            pattern={Lines[angle=25,distance=3pt,line width=.38pt]},
            pattern color=angularorange,
            draw=angularorange!85!black,
            line width=.35pt
        },
        qtwo/.style={
            pattern={Lines[angle=70,distance=3pt,line width=.38pt]},
            pattern color=angularblue,
            draw=angularblue!85!black,
            line width=.35pt
        },
        qthree/.style={
            pattern={Lines[angle=115,distance=3pt,line width=.38pt]},
            pattern color=angulargreen,
            draw=angulargreen!85!black,
            line width=.35pt
        },
        qfour/.style={
            pattern={Lines[angle=160,distance=3pt,line width=.38pt]},
            pattern color=angularpurple,
            draw=angularpurple!85!black,
            line width=.35pt
        },
        qlabel/.style={
            font=\scriptsize,
            fill=white,
            fill opacity=.86,
            text opacity=1,
            inner sep=1pt
        }
    }

    \begin{scope}[shift={(-3.45,0)}]
        \draw[->,gray!70] (-1.55,0)--(1.72,0) node[right,black] {$x_1$};
        \draw[->,gray!70] (0,-1.55)--(0,1.72) node[above,black] {$x_2$};
        \draw[black!70,line width=.55pt] (0,0) circle (1.18);
        \foreach \angle in {0,90,180,270}{
            \draw[->,black!45,line width=.45pt] (0,0)--(\angle:1.12);
            \fill[black] (\angle:1.18) circle (2.15pt);
        }
        \node[anchor=south west] at (0:1.18) {$\eps_1$};
        \node[anchor=south east] at (180:1.18) {$-\eps_1$};
        \node[anchor=south west] at (90:1.18) {$\eps_2$};
        \node[anchor=north east] at (270:1.18) {$-\eps_2$};
        \node[font=\small] at (0,-2.02) {support of the cylindrical measure:};
        \node[font=\scriptsize] at (0,-2.4)
            {$\operatorname{supp}\mu=\{\pm \eps_1,\pm \eps_2\}$};
    \end{scope}

    \begin{scope}[shift={(2.15,0)}]
        \def\innerradius{1.02}
        \def\outerradius{1.86}
        \draw[->,gray!65] (-2.12,0)--(2.28,0) node[right,black] {$\xi_1$};
        \draw[->,gray!65] (0,-2.12)--(0,2.28) node[above,black] {$\xi_2$};

        \path[qone]
            (-28:\innerradius) arc[start angle=-28,end angle=28,radius=\innerradius]
            -- (28:\outerradius) arc[start angle=28,end angle=-28,radius=\outerradius]
            -- cycle;
        \path[qone]
            (152:\innerradius) arc[start angle=152,end angle=208,radius=\innerradius]
            -- (208:\outerradius) arc[start angle=208,end angle=152,radius=\outerradius]
            -- cycle;

        \path[qtwo]
            (17:\innerradius) arc[start angle=17,end angle=73,radius=\innerradius]
            -- (73:\outerradius) arc[start angle=73,end angle=17,radius=\outerradius]
            -- cycle;
        \path[qtwo]
            (197:\innerradius) arc[start angle=197,end angle=253,radius=\innerradius]
            -- (253:\outerradius) arc[start angle=253,end angle=197,radius=\outerradius]
            -- cycle;

        \path[qthree]
            (62:\innerradius) arc[start angle=62,end angle=118,radius=\innerradius]
            -- (118:\outerradius) arc[start angle=118,end angle=62,radius=\outerradius]
            -- cycle;
        \path[qthree]
            (242:\innerradius) arc[start angle=242,end angle=298,radius=\innerradius]
            -- (298:\outerradius) arc[start angle=298,end angle=242,radius=\outerradius]
            -- cycle;

        \path[qfour]
            (107:\innerradius) arc[start angle=107,end angle=163,radius=\innerradius]
            -- (163:\outerradius) arc[start angle=163,end angle=107,radius=\outerradius]
            -- cycle;
        \path[qfour]
            (287:\innerradius) arc[start angle=287,end angle=343,radius=\innerradius]
            -- (343:\outerradius) arc[start angle=343,end angle=287,radius=\outerradius]
            -- cycle;

        \draw[black!75,line width=.55pt] (0,0) circle (\innerradius);
        \draw[black!75,line width=.55pt] (0,0) circle (\outerradius);

        \draw[->,angularorange!85!black,line width=.65pt]
            (0,0)--(0:.78)
            node[anchor=south west,font=\scriptsize,inner sep=.8pt] {$e_1$};
        \draw[->,angularblue!85!black,line width=.65pt]
            (0,0)--(45:.78)
            node[anchor=south west,font=\scriptsize,inner sep=.8pt] {$e_2$};
        \draw[->,angulargreen!80!black,line width=.65pt]
            (0,0)--(90:.78)
            node[anchor=south east,font=\scriptsize,inner sep=.8pt] {$e_3$};
        \draw[->,angularpurple!85!black,line width=.65pt]
            (0,0)--(135:.78)
            node[anchor=south east,font=\scriptsize,inner sep=.8pt] {$e_4$};

        \node[qlabel,text=angularorange!80!black] at (0:1.45) {$q_1$};
        \node[qlabel,text=angularorange!80!black] at (180:1.45) {$q_1$};
        \node[qlabel,text=angularblue!85!black] at (45:1.45) {$q_2$};
        \node[qlabel,text=angularblue!85!black] at (225:1.45) {$q_2$};
        \node[qlabel,text=angulargreen!80!black] at (90:1.45) {$q_3$};
        \node[qlabel,text=angulargreen!80!black] at (270:1.45) {$q_3$};
        \node[qlabel,text=angularpurple!85!black] at (135:1.45) {$q_4$};
        \node[qlabel,text=angularpurple!85!black] at (315:1.45) {$q_4$};

        \draw[black!60,line width=.4pt]
            (28:\outerradius)--++(.30,.25)
            node[right,black,font=\scriptsize,fill=white,inner sep=1pt]
            {$C_0 2^j$};
        \draw[black!60,line width=.4pt]
            (-28:\innerradius)--++(.32,-.26)
            node[right,black,font=\scriptsize,fill=white,inner sep=1pt]
            {$c_0 2^j$};
        \node[font=\small] at (0,-2.48) {angular decomposition in $\xi$-space:};
        \node[font=\scriptsize,align=center] at (0,-3)
            {same-colored opposite sectors correspond to one $q_m$};
    \end{scope}
    \end{tikzpicture}
    \caption{
    The black dots show the four directions in
    $\operatorname{supp}\mu=\{\pm \eps_1,\pm \eps_2\}$. Let
    $A_j:=\{\xi:c_0 2^j\leq |\xi|\leq C_0 2^j\}$. For each $m$, the
    colored arrow in the central disk indicates the representative direction
    $e_m$, while the two opposite sectors with the same color and hatching show
    $A_j\cap\{\xi:\xi/|\xi|\in\operatorname{supp}q_m\}$. The Fourier support
    of $P_{j,m}h$ is contained in this set. Neighboring sectors overlap
    because the $q_m$ form a smooth partition of unity. The two sectors are
    centered at $\pm e_m$ and are opposite because $q_m$ is even.}
    \label{fig:conic}
\end{figure}

The following lemma is inspired by \cite[Lemma 3.4]{WZ11frequency} and provides an angular analogue. 
\begin{lemma}[Angular dyadic maximum principle]\label{lem:angular-dmp}
Let $w \in \sS'(\mR^d)$ and $v=P_{j,m}w$. Suppose $v$ attains a positive maximum at
$x_0$, namely
\[
    v(x_0)=\|v\|_{L^\infty}>0.
\]
Then
\begin{equation*}
    -Lv(x_0)\geq c2^{\alpha j}v(x_0),
\end{equation*}
where $c>0$ depends only on
$d,\alpha,\kappa,M_\mu$ and on the fixed cutoffs.
\end{lemma}

\begin{proof}
For the stable operator $L$ defined above,
\[
    -L\varphi(x)
    :=\int_{\mS^{d-1}}\int_0^\infty
    \left[
        \varphi(x)-\varphi(x+r\theta)
    \right]\frac{\d r}{r^{1+\alpha}}\,\mu(\d\theta).
\]
At the maximum point $x_0$, $\nabla v(x_0)=0$. Fix $\theta\in E(e_m)$ and set
\[
    g_\theta(r):=v(x_0+r\theta).
\]
By Lemma \ref{lem:support}, the one-dimensional Fourier support of
$g_\theta$ is contained in the projection of $\operatorname{supp}\widehat v$ onto the line
spanned by $\theta$. Using \eqref{eq:cone-pro} and \eqref{eq:support}, we get
\[
    \operatorname{supp} \widehat{g_\theta}
    \subset
    \{\eta\in\mR: a2^j\leq |\eta|\leq A2^j\},
\]
where $a=c_0\delta$ and $A=C_0$ are independent of $j,m$ and $\theta$.
Since $x_0$ is a global maximum of $v$,
\[
    g_\theta(0)=v(x_0)=\|v\|_{L^\infty}.
\]
By Lemma \ref{lem:1d-dmp},
\begin{equation}\label{eq:v-1d}
    \int_0^\infty \frac{v(x_0)-v(x_0+r\theta)}{r^{1+\alpha}}\,\d r
    \geq c2^{\alpha j}v(x_0),
    \quad \theta\in E(e_m).
\end{equation}
Since $x_0$ is a maximum point, the integrand
$v(x_0)-v(x_0+r\theta)$ is non-negative for all $r>0$ and
$\theta\in\mS^{d-1}$. Hence, integrating \eqref{eq:v-1d} over
$E(e_m)$ and using \eqref{eq:E-lower}, we obtain
\[
    -Lv(x_0)
    \geq
    \int_{E(e_m)}\int_0^\infty
    \frac{v(x_0)-v(x_0+r\theta)}{r^{1+\alpha}}\,\d r\,\mu(\d\theta)
    \geq c2^{\alpha j}v(x_0).
\]
This proves the lemma.
\end{proof}

\section{Proof of Theorem \ref{thm:main}}

\begin{proof}
Apply $P_{j,m}$ to \eqref{eq:PDE}. Let
\[
    v=P_{j,m}u,
    \quad
    G_{j,m}:=P_{j,m}h - [b\cdot\nabla,P_{j,m}]u.
\]
Since $(-L)$ commutes with $P_{j,m}$, then
\begin{equation}\label{eq:v-eq}
    \p_t v+\lambda v -Lv - b\cdot\nabla v=G_{j,m} 
\end{equation}
with initial condition $v(0,\cdot)=P_{j,m}f$. 

We first claim that for every $j\in \mZ$ and $1\leq m\leq N$,
\begin{equation}\label{eq:dyadic-est}
    \|P_{j,m}u\|_{L^\infty_T L^\infty_x}
    \leq
    \|P_{j,m}f\|_{L^\infty(\mR^d)}
    +(\lambda+c2^{\alpha j})^{-1}
    \|G_{j,m}\|_{L^\infty_T L^\infty_x}.
\end{equation}

In fact, assume first that $v$ attains a positive space-time maximum at
$(t_0,x_0)$. If this maximum is no larger than
$\|P_{j,m}f\|_{L^\infty}$, it is already controlled by the initial datum.
Otherwise $t_0>0$. At such a maximum point, using the one-sided time
derivative if $t_0=T$,
\[
    \p_t v(t_0,x_0)\geq0,
    \quad
    \nabla v(t_0,x_0)=0.
\]
By Lemma \ref{lem:angular-dmp},
\[
    (-L)v(t_0,x_0)\geq c2^{\alpha j}v(t_0,x_0).
\]
Evaluating \eqref{eq:v-eq} at $(t_0,x_0)$ yields
\[
    (\lambda+c2^{\alpha j})v(t_0,x_0)
    \leq
    \|G_{j,m}\|_{L^\infty_T L^\infty_x}.
\]
Applying the same argument to $-v$ at a negative minimum gives
\eqref{eq:dyadic-est}. 

If maxima or minima are not attained, one can use standard cut-off argument, as presented in the proof of \cite[Theorem 3.6]{LZ22nonlocal}, to obtain the same estimate.

\medskip

Now let $K_{j,m}$ denote the convolution kernel of $P_{j,m}$. It is not hard to verify that 
\begin{equation}\label{eq:kernel}
\begin{aligned}
    \int_{\mR^d}K_{j,m}(y)\,\d y=0,
    \quad
    &\int_{\mR^d}yK_{j,m}(y)\,\d y=0,
    \\
    \int_{\mR^d}|y|^\beta |K_{j,m}(y)|\,\d y
    \leq C2^{-\beta j},
    \quad
    &\int_{\mR^d}|y|^s |K_{j,m}(y)|\,\d y
    \leq C2^{-sj}.
\end{aligned}
\end{equation}
Noting that 
\[
    P_{j,m}f(x)
    =\int_{\mR^d}K_{j,m}(y)
    \bigl(f(x-y)-f(x)+y\cdot\nabla f(x)\bigr)\,\d y, 
\]
we have 
\begin{equation}\label{eq:Pjm-initial-estimate}
    \|P_{j,m}f\|_{L^\infty(\mR^d)}
    \leq
    C2^{-sj}[f]_{\dot{C}^s}.
\end{equation}

Similarly, 
\[
    P_{j,m}h(t,x)
    =\int_{\mR^d}K_{j,m}(y)
    \bigl(h(t,x-y)-h(t,x)\bigr)\,\d y
\]
gives 
\begin{equation}\label{eq:Pjm-h-estimate}
    \|P_{j,m}h\|_{L^\infty_T L^\infty_x}
    \leq
    C2^{-\beta j}[h]_{L_T^\infty \dot{C}_x^\beta}.
\end{equation}

For the commutator, we write
\[
\begin{aligned}
    [b\cdot\nabla,P_{j,m}]u(t,x)
    &=b(t,x)\cdot\nabla P_{j,m}u(t,x)
      -P_{j,m}\bigl(b(t,\cdot)\cdot\nabla u(t,\cdot)\bigr)(x) \\
    &=\int_{\mR^d}K_{j,m}(y)
    \bigl(b(t,x)-b(t,x-y)\bigr)\cdot\nabla u(t,x-y)\,\d y.
\end{aligned}
\]
Again using \eqref{eq:kernel}, we obtain
\begin{equation}\label{eq:comm-estimate}
    \|[b\cdot\nabla,P_{j,m}]u\|_{L^\infty_T L^\infty_x}
    \leq
    C2^{-\beta j}
    [b]_{L_T^\infty \dot{C}_x^\beta}
    \|\nabla u\|_{L^\infty_T L^\infty_x}.
\end{equation}
Combining \eqref{eq:dyadic-est},
\eqref{eq:Pjm-initial-estimate}, \eqref{eq:Pjm-h-estimate} and
\eqref{eq:comm-estimate}, and using
$\lambda+c2^{\alpha j}\geq c2^{\alpha j}$, gives
\begin{equation*}
    \|P_{j,m}u\|_{L^\infty_T L^\infty_x}
    \leq
    C2^{-(\alpha+\beta)j}
    \left(
        [f]_{\dot{C}^s}
        +[h]_{L_T^\infty \dot{C}_x^\beta}
        +[b]_{L_T^\infty \dot{C}_x^\beta}
        \|\nabla u\|_{L^\infty_T L^\infty_x}
    \right).
\end{equation*}
By \eqref{eq:sum-Pjm},
\begin{equation*}
    \|\dot{\Delta}_j u\|_{L^\infty_T L^\infty_x}
    \leq
    C2^{-sj}
    \left(
        \|f\|_{C^s}
        +[h]_{L_T^\infty \dot{C}_x^\beta}
        +[b]_{L_T^\infty \dot{C}_x^\beta}
        \|\nabla u\|_{L^\infty_T L^\infty_x}
    \right).
\end{equation*}
Taking the supremum over $j \in \mZ$ and using the Littlewood--Paley
characterization of $\dot{C}^s$ for $s\in(1,2)$, we get
\begin{equation*}
    [u]_{L_T^\infty \dot{C}_x^s}
    \leq
    C\left(
        [f]_{\dot{C}^s}
        +[h]_{L_T^\infty \dot{C}_x^\beta}
        +[b]_{L_T^\infty \dot{C}_x^\beta}
        \|\nabla u\|_{L^\infty_T L^\infty_x}
    \right).
\end{equation*}
This proves \eqref{eq:main-pre}.

On the other hand, we also have the parabolic maximum principle
\begin{equation}\label{eq:mp}
    \|u\|_{L^\infty_T L^\infty_x}
    \leq
    \|f\|_{L^\infty(\mR^d)}
    +\lambda^{-1}\|h\|_{L^\infty_T L^\infty_x}.
\end{equation}
In fact, if a positive space-time maximum of $u$ is attained at $t=0$, it is
bounded by $\|f\|_{L^\infty}$. Otherwise, at the maximum point,
\[
    \p_tu\geq0,
    \quad
    \nabla u=0,
    \quad
    (-L)u\geq0.
\]
Hence $\lambda u\leq h$ at that point. Applying the same argument to $-u$
gives the lower bound.

\medskip 

Since $s\in(1,2)$, the interpolation inequality
\begin{equation}\label{eq:interp}
    \|\nabla u\|_{L^\infty_T L^\infty_x}
    \leq
    \eps [u]_{L_T^\infty \dot{C}_x^s}
    +C\eps^{-1/(s-1)}\|u\|_{L^\infty_T L^\infty_x}
\end{equation}
holds for every $\eps\in(0,1)$.

Substituting \eqref{eq:mp} and \eqref{eq:interp} into
\eqref{eq:main-pre} gives
\[
\begin{aligned}
    [u]_{L_T^\infty \dot{C}_x^s}
    &\leq
    C[f]_{\dot{C}^s}
    +C[h]_{L_T^\infty \dot{C}_x^\beta}
    +C\eps[b]_{L_T^\infty \dot{C}_x^\beta}
    [u]_{L_T^\infty \dot{C}_x^s} \\
    &\quad
    +C \eps^{-1/(s-1)} [b]_{L_T^\infty \dot{C}_x^\beta}
    \lambda^{-1}
    \|h\|_{L^\infty_T L^\infty_x} +C\eps^{-1/(s-1)}[b]_{L_T^\infty \dot{C}_x^\beta}
    \|f\|_{L^\infty(\mR^d)}.
\end{aligned}
\]
Choose $\eps>0$ sufficiently small, so that
\[
    C\eps[b]_{L_T^\infty \dot{C}_x^\beta}\leq \frac{1}{2}.
\]
Then
\begin{equation*}
    [u]_{L_T^\infty \dot{C}_x^s}
    \leq
    C\left(1+[b]_{L_T^\infty \dot{C}_x^\beta}^{s/(s-1)}\right)\|f\|_{C^s}
    +C\left(
        1+\lambda^{-1}[b]_{L_T^\infty \dot{C}_x^\beta}^{s/(s-1)}
    \right)\|h\|_{L_T^\infty C_x^\beta}.
\end{equation*}
Combining this and \eqref{eq:mp} gives \eqref{eq:main-final}. The proof is complete.
\end{proof}

\bibliographystyle{alpha}
\bibliography{mybib.bib}

\end{document}